\documentclass[12pt,centertags]{amsart}
\usepackage[usenames,dvipsnames]{xcolor}
\usepackage[colorlinks,linkcolor=blue,citecolor=blue,bookmarksnumbered,final]{hyperref}

\usepackage{amsmath,amstext,amsthm,a4,amssymb,amscd,mathrsfs}
\usepackage{mathtools}
\usepackage[mathscr]{eucal}
\usepackage{mathrsfs}
\usepackage{epsf}
\usepackage{enumitem}
\usepackage[normalem]{ulem}

\numberwithin{equation}{section}

\def\mE{\mathcal{E}}
\def\mF{\mathcal{F}}
\def\mH{\mathcal{H}}

\def\mM{\mathcal{M}}
\def\mN{\mathcal{N}}

\def\mW{\mathcal{W}}

\newcommand{\dd}{\mathrm{d}}

\newtheorem{thm}{Theorem}[section]
\newtheorem{lemma}[thm]{Lemma}
\newtheorem{prop}[thm]{Proposition}

\theoremstyle{definition}

\theoremstyle{definition}

\theoremstyle{definition}
\newtheorem{defn}[thm]{Definition}
\newcommand{\be}{\begin{eqnarray}}
\newcommand{\ee}{\end{eqnarray}}

\newcommand{\comment}[1]{}
\DeclareMathOperator{\kccw}{K-ccw_2}
\DeclareMathOperator{\kcw}{K-cw_2}

\DeclareMathOperator{\acw}{\widehat{A}-ccw_2}
\DeclareMathOperator{\aar}{\widehat{A}-area}

\begin{document}

\title{$\mathrm{K}$-cowaist on complete foliated manifolds}
 
\author{Guangxiang Su and Xiangsheng Wang}
\address{Chern Institute of Mathematics \& LPMC, Nankai
University, Tianjin 300071, P.R. China}
\email{guangxiangsu@nankai.edu.cn}
\address{School of Mathematics, Shandong University, Jinan, Shandong 250100, P.R. China}
\email{xiangsheng@sdu.edu.cn}

\begin{abstract}  
Let $(M,F)$ be a connected (not necessarily compact) foliated manifold carrying a complete Riemannian metric $g^{TM}$. We generalize Gromov's $\mathrm{K}$-cowaist using the coverings of $M$, as well as defining a closely related concept called the $\widehat{\mathrm{A}}$-cowaist. Let $k^F$ be the associated leafwise scalar curvature of $g^F = g^{TM}|_F$. We obtain some estimates on $k^F$ using these two concepts. In particular, assuming that the generalized $\mathrm{K}$-cowaist is infinity and either $TM$ or $F$ is spin, we show that $\inf(k^F)\leq 0$.
\end{abstract}

\maketitle

\section{Introduction} \label{s0}

\subsection{Main results}

Let $M$ be a closed connected oriented smooth Riemannian manifold. Let $E\to M$ be a Hermitian vector bundle with a Hermitian connection $\nabla^E$ and $R^E$ be the curvature of $\nabla^E$. If $\dim M$ is even, Gromov (\cite[\S 4]{Gr96} or~\cite[\S 4.1.4]{Gr}) defines the $\mathrm{K}$-cowaist\footnote{In~\cite{Gr96}, $\mathrm{K}$-cowaist was called $\mathrm{K}$-area. But recently, Gromov~\cite{Gr} suggests that $\mathrm{K}$-cowaist should be a more proper name for this concept.} of $M$ by
\begin{equation*}
  \kcw(M)=\sup_{E}(\|R^E\|^{-1}),
\end{equation*}
where $E\to M$ is a unitary bundle for which (at least) one characteristic (Chern) number of $E$ does not vanish. Gromov also generalizes the definition of the $\mathrm{K}$-cowaist to open manifolds by sticking to bundles $E\to M$ trivialized at infinity and using the characteristic numbers coming from the cohomology with compact supports. Moreover, if $\dim M$ is odd, Gromov defines
\begin{equation*}
  \kcw(M)=\sup_k \kcw(M\times \mathbb{R}^k),
\end{equation*}
where one takes those $k\geq 0$ such that $\dim M+k$ is even.

The $\mathrm{K}$-cowaist is closely related to the scalar curvature.
In \cite[\S 5$1\over 4$]{Gr96}, Gromov proves that every complete Riemannian spin manifold of dimension $n$ with the scalar curvature $k^{g^{TM}}\geq \varepsilon^{-2}$ satisfies
\begin{equation*}
\kcw(M)\leq {\rm const}_n \varepsilon^2.
\end{equation*}

In \cite[\S 9$2\over 3$]{Gr96}, Gromov also defines the $\mathrm{K}$-cowaist for foliated manifolds.
In this paper, for the case that $M$ is a foliated manifold again, we further generalize the definition of $\mathrm{K}$-cowaist by considering the coverings of $M$.
As in~\cite{Gr96}, we also
study the relation between the leafwise scalar curvature and this generalized $\mathrm{K}$-cowaist.

We now explain it in detail.
Let $M$ be a connected oriented (not necessarily compact) manifold carrying a {(not necessarily complete)} Riemannian metric $g^{TM}$. Let $F\subseteq TM$ be an integrable subbundle of $TM$ and $g^F=g^{TM}|_F$ be the restricted metric on $F$. In the following, we assume that both $\dim M$ and ${\rm rk}(F)$ are even.
If $\dim M$ is odd and ${\rm rk}(F)$ is even, we replace $M$ by $M\times S^1$. If ${\rm rk}(F)$ is odd, we replace $F$ by $F\oplus TS^1$ and $M$ by $M\times S^1\times S^1$ or $M\times S^1$ depending on whether $\dim M$ is even or odd.

We take $\widetilde{\pi}:\widetilde{M}\to M$ to be a covering of $M$ and $\widetilde{F}$ to be the lifted foliation on $\widetilde{M}$. Then $\widetilde{M}$ and $\widetilde{F}$ carry the lifted metrics $g^{T\widetilde{M}}$ and $g^{\widetilde{F}}$.

Let $(E,g^E,\nabla^E)$ be a Hermitian vector bundle over $\widetilde{M}$ with the Hermitian metric $g^E$ and the Hermitian connection $\nabla^E$. We assume that $E$ is trivial at infinity.

Let $R^{E}=(\nabla^{E})^2$ be the curvature of $\nabla^{E}$. Hence, for any $x\in \widetilde{M}$ and $\alpha,\beta\in T_x\widetilde{M}$, $R^{E}(\alpha\wedge \beta)\in \mathrm{End}(T_x\widetilde{M}).$
Recall that $\|R^{{E}}_{\widetilde{F}}\|$ is defined by\footnote{In the literature, there are different ways to define the norm of $R^{E}$ (or $R^{{E}}_{\widetilde{F}}$), eg~\cite{Baer_2023bo} and~\cite{Listing_2013ho}. Here, we adopt the same definition as in~\cite{Wang22}. However, for the content discussed in this paper, all these norms work.} (cf. \cite{Gr96}),
\begin{equation}
  \label{eq:def-rn}
  \|R^{{E}}_{\widetilde{F}}\| = \sup_{x\in \widetilde{M}} \sup_{\substack{\alpha,\beta\in \widetilde{F}_x \\ \alpha \perp \beta,\ |\alpha\wedge \beta|=1}}|R^{E}(\alpha\wedge \beta)|.
\end{equation}

Now, we can define a pair of closely related concepts.
\begin{defn}
  With above notation, if the vector bundle $E$ satisfies
some Chern number of $E$ is nonzero, we define the (\emph{covering}) \emph{$\mathrm{K}$-cowaist} of $(M,F)$ by 
  \begin{align}
    \kccw(M,F)=\sup _{\widetilde{M},E}(\|R^{{E}}_{\widetilde{F}}\|^{-1}).
  \end{align}
Similarly, if the vector bundle $E$ satisfies the inequality
  \begin{equation*}
    \int_{\widetilde{M}}\widehat{\mathrm{A}}(T\widetilde{M})({\rm ch}(E)-{\rm rk}(E))\neq 0,
  \end{equation*}
  we define the (\emph{covering}) \emph{$\widehat{\mathrm{A}}$-cowaist} of $(M,F)$ by
  \begin{align}
    \acw(M,F)=\sup _{\widetilde{M},E}(\|R^{{E}}_{\widetilde{F}}\|^{-1}).
  \end{align}
\end{defn}

From the index theory viewpoint, among the two concepts given as above,
the $\widehat{\mathrm{A}}$-cowaist perhaps relates to the scalar curvature more
directly. In fact, we can use $\widehat{\mathrm{A}}$-cowaist to give a
quantitative estimate on the leafwise scalar curvature.

\begin{thm}
 \label{thm:4.1}
  Let $M$ be a connected oriented (not necessarily compact) manifold carrying a complete Riemannian metric $g^{TM}$. Let $F\subseteq TM$ be an integrable subbundle of $TM$ with the restricted metric $g^F=g^{TM}|_{F}$. Let $k^F$ be the associated leafwise scalar curvature of $F$. If either $TM$ or $F$ is spin, then
  \begin{equation*}\label{4.3r}
        \inf(k^F)\leq \frac{2{\rm rk}(F)({\rm rk}(F)-1)}{ \acw(M,F)}.
  \end{equation*}
\end{thm}

Note that in the definition of $\widehat{\mathrm{A}}$-cowaist, we don't need
that $g^{TM}$ is complete. But for the above theorem, the completeness is
necessary.
As in \cite{Z17}, we only give the proof of Theorem \ref{thm:4.1} for the $TM$ spin case in detail. The $F$ spin case can be proved similarly as in \cite[\S 2.5]{Z17}.

To obtain a similar estimate using the $\mathrm{K}$-cowaist, we notice the following reinterpretation of the result in~\cite[\S~5$\frac{3}{8}$]{Gr96}.
One can also see~\cite{Baer_2023bo},~\cite{Listing_2013ho} or~\cite{Wang22} for a more detailed proof of this result.
\begin{prop}
  $\kccw(M,F) = +\infty$ implies $\acw(M,F) = +\infty$.
\end{prop}

Due to this proposition, we have the following corollary of Theorem~\ref{thm:4.1}.
\begin{thm}
  \label{thm:ka}
  Under the same assumptions of Theorem~\ref{thm:4.1}, if we further assume that $\kccw(M, F) = +\infty$, then $\inf(k^F) \le 0$.
\end{thm}

Since $\kcw(M)=+\infty$ implies $\kccw(M,F)=+\infty$, as a corollary of Theorem \ref{thm:ka}, one can show the following result.
\begin{thm}{\rm (Gromov,~\cite[p.~258, footnote~277]{Gr})}
Complete manifolds $X$ with infinite $\kcw(X)$, carry no spin foliations $F$, where the induced Riemannian metrics in the leaves satisfy $k^F\geq \sigma>0$.
\end{thm}

\subsection{A discussion about the definitions of $\widehat{\mathrm{A}}$-cowaist}
Compared to other similar concepts in the literature, a feature of the definition of $\widehat{\mathrm{A}}$-cowaist is that the supremum is calculated using bundles over any coverings of $M$ rather than bundles over $M$ alone.
Whether it is necessary to take the supremum on this larger set turns out to be a delicate problem.
For simplicity, we will assume $F = TM$ in this subsection.

To facilitate our discussion, we use the following definition which resembles the definition of $\widehat{\mathrm{A}}$-cowaist but dose not use the covering space.
Let $M$ be a manifold and $(E,g^E,\nabla^E)$ be a Hermitian vector bundle over $M$ with the Hermitian metric $g^E$ and the Hermitian connection $\nabla^E$.
If $M$ is non-compact, we also assume that $E$ is trivial at infinity.
Define
\begin{equation*}
  \aar(M) = \sup _{E}\Big\{\|R^{{E}}\|^{-1}| \textstyle\int_{{M}}\widehat{\mathrm{A}}(T{M})(\mathrm{ch}(E)-\mathrm{rk}(E))\neq 0\Big\}.
\end{equation*}

We discuss several cases separately.

\subsubsection*{Case 1: $M$ is non-compact.}
In this case, $\acw(M,TM)$ in strictly larger than $\aar(M)$ in general.
As a example, let $M$ be an annulus in $\mathbb{R}^2$ and $F = TM$.
By \cite[p.~33~(c)]{Gr101}, we have
\begin{equation}
  \label{eq:area-M}
  \kcw(M) = \mathrm{area}(M).
\end{equation}
Since, $\dim M = 2$, for any vector bundle $E$ over $M$, we have
\begin{equation*}
\int_{{M}}\widehat{\mathrm{A}}(T{M})(\mathrm{ch}(E)-\mathrm{rk}(E)) = \int_M \mathrm{c}_1(E).
\end{equation*}
Hence, $\kcw(M)$ and $\aar(M)$ coincide in this case.
Using (\ref{eq:area-M}), we can see that
\begin{equation*}
  \acw(M,TM)=\sup _{\widetilde{M}} \mathrm{area}(\widetilde{M}) = +\infty > \aar(M).
\end{equation*}

\subsubsection*{Case 2: $M$ is compact and the universal covering of $M$ is compact}
First, we note that if $M$ is compact, $\aar(M)$ is just a small variant of~\cite[Definition~1.6]{CZ21}.\footnote{In fact, if $M$ carries a metric with the positive scalar curvature, these two definitions are the same.}

Since the universal covering of $M$ is compact, any covering space of $M$, $\widetilde{M}$, is also compact.
Then by the same proof of~\cite[p.~33~(d)]{Gr101}, we have
\begin{equation*}
  \aar(M) = \aar{(\widetilde{M})}.
\end{equation*}
Therefore, in this case, $\acw(M,TM)$ is equal to $\aar(M)$.

\subsubsection*{Case 3: $M$ is compact and the universal covering of $M$ is non-compact}
This case is the most difficult and we do not have a definite answer at the moment.
In fact, whether $\aar(M)$ is equal to $\acw(M,TM)$ in this case relates closely to Gromov's question~\cite[p.~34, Question~23]{Gr101}: is there a closed manifold $M$ such that $\kcw(M) < \infty$ and the universal covering of $M$ satisfies $\kcw(\widetilde{M}) = \infty$?

The main difficulty in this case is that the pull-back and push-forward construction for vector bundles do not work well for the non-compact spaces.
If we put some restrictions on the covering spaces, maybe some partial results are still possible.
For example, if we assume $\pi_1(M)$ is residually finite, we have the following simple extension of~\cite[p.~26, (v')]{Gr96}.

\begin{prop}
  If $M$ is a closed manifold and $\pi_1(M)$ is residully finite, for the universal covering space $\widetilde{M}$ of $M$,
  \begin{equation*}
\kccw(M)\ge \kcw(M) \ge \kcw{(\widetilde{M})} = \kccw(\widetilde{M}).
  \end{equation*}
\end{prop}
\begin{proof}
  Clearly, we only need to show that
  \begin{equation}
    \label{eq:pf}
    \kcw(M) \ge \kcw{(\widetilde{M})}
  \end{equation}
  If $\widetilde{M}$ is compact, (\ref{eq:pf}) follows from the push-forward inequality~\cite[\S~4$\frac{3}{5}$]{Gr96}.
  We try to show that the push-forward argument also works for non-compact $\widetilde{M}$.

  Take a vector bundle $E$ over $\widetilde{M}$ which is trivial outside a compact set $K$.
  Since $\pi_1(M)$ is residully finite, we can find a finite covering of $M$, $N$, such that the covering map from $\widetilde{M}$ to $N$ is injective on $K$.
  As a result, we can push forward the vector bundle $E$ to a vector bundle $E_N$ over $N$.
  Since $N$ is a finite covering of $M$, we have
  \begin{equation*}
    \kcw(N) = \kcw(M).
  \end{equation*}
  Therefore,
  \begin{equation*}
    \|R^{{E}}\| ^{-1}= \|R^{{E_N}}\|^{-1} \le \kcw(N) = \kcw(M),
  \end{equation*}
  from which we obtain (\ref{eq:pf}).
\end{proof}

We also note that the $\mathrm{K}$-cowaist also generalizes \cite[Definition 5.1]{BH19}.

\section{Proof of Theorem \ref{thm:4.1}}
\setcounter{equation}{0}

In this section we prove Theorem \ref{thm:4.1}.
Our strategy follows the proof of~\cite[Theorem~1.2]{SWZ} closely. Note that in~\cite{SWZ}, there is a map $f:M \rightarrow S^n(1)$, which enables us to construct suitable bundles over a closed manifold associated with the non-compact manifold $M$.
Compared to~\cite{SWZ}, the new idea in this paper is that we will show that in the current situation, the auxiliary bundles needed for the proof, as well as the endomorphisms between the these bundles, can be constructed \emph{without using the map $f$}.

We argue by contradiction.
Assume that
\begin{equation*}
    \inf(k^F)> {\frac{2{\rm rk}(F)({\rm rk}(F)-1)}{\acw(M,F)}}.
\end{equation*}
  Then, by the definition of $\acw(M,F)$, there exists
  \begin{itemize}
  \item a covering manifold $\widetilde{\pi}:\widetilde{M}\to M$ with the lifted foliation $\widetilde{F}$ and the lifted metrics $g^{T\widetilde{M}}$ and $g^{\widetilde{F}}$,

  \item a Hermitian vector bundle $E_0$ over $\widetilde{M}$ with the Hermitian metric $g^{E_0}$ and the Hermitian connection $\nabla^{E_0}$, which is trivial at infinity and satisfies
  \begin{align}\label{84}
    \int_{\widetilde{M}}\widehat{\mathrm{A}}(T\widetilde{M})\left({\rm ch}(E_0)-{\rm rk}(E_0)\right)\neq 0,
  \end{align}

  \item a constant $\kappa>0$ such that
    \begin{align}\label{4.3}
      \widetilde{\pi}^*(k^F)-2{\rm rk}(F)\left({\rm rk}(F)-1\right)\|R^{E_0}_{\widetilde{F}}\|> \kappa \text{ on }\widetilde{M}.
    \end{align}
   
  \end{itemize}

As explained in Introduction, we only prove the $TM$ spin case in detail. In the following, we assume that $TM$ is spin.

To begin, we note that if both $M$ and $\widetilde{M}$ are compact, by \cite[Section 1.1]{Z19}, one gets a contradiction easily.
Therefore, in the following, we assume that $\widetilde{M}$ is noncompact.
For the rest of the proof we will only deal with quantities associated with $\widetilde{M}$ and $\widetilde{F}$.
To simplify the notation, we will denote the foliation $(\widetilde{M},\widetilde{F})$ by $(M,F)$ and the metrics $(g^{T\widetilde{M}},g^{\widetilde{F}})$ by $(g^{TM},g^F)$.

Roughly speaking, we will prove Theorem~\ref{thm:4.1} in three steps.
\begin{enumerate}[label=(\roman*)]
\item We construct a closed manifold $\widehat{M}_{H_{3m}}$ and a $\mathbf{Z}_2$-graded bundle $\widehat{E}$ over it. We also construct a fiber bundle $\widehat{\mM}_{\mH_{3m},R}$ over $\widehat{M}_{H_{3m}}$ associated with the foliation $F$.
\item We construct a deformed Dirac operator on $\widehat{\mM}_{\mH_{3m},R}$ and obtain some estimates about it.
\item We construct a closed manifold using $\widehat{\mM}_{\mH_{3m},R}$ and an operator $P^{\mE_{3m,R}}_{R,\beta,\gamma,+}$ using the deformed Dirac operator. We will show that (\ref{84}) implies the index of $P^{\mE_{3m,R}}_{R,\beta,\gamma,+}$ is not zero while (\ref{4.3}) implies the index of $P^{\mE_{3m,R}}_{R,\beta,\gamma,+}$ is zero. Thus we obtain a contradiction.
\end{enumerate}

As we have said at the beginning of this section, the main difference between the proof of Theorem~\ref{thm:4.1} and~\cite[Theorem~1.2]{SWZ} is the first step.

\subsubsection*{Step 1.}
Let $(E_1={M}\times {\bf C}^k,g^{E_1},\nabla^{E_1})$, with $k={\rm rk}(E_0)$, be the trivial vector bundle on ${M}$.
Then, let $E=E_0\oplus E_1$ be a ${\bf Z}_2$-graded Hermitian vector bundle over ${M}$ with a ${\bf Z}_2$-graded metric $g^E=g^{E_0}\oplus g^{E_1}$ and a ${\bf Z}_2$-graded Hermitian connection $\nabla^{E}=\nabla^{E_0}\oplus \nabla^{E_1}$.

Since $(E_0,g^{E_0},\nabla^{E_0})$ is trivial at infinity, there exists a compact subset $K$\footnote{We can and will choose $K$ is a closure of an open subset.} and an isomorphism
\begin{equation}
  \label{eq:def-psi}
  \psi: (E_0|_{{M} \setminus K}, g^{E_0}, \nabla^{E_0}) \rightarrow (({M} \setminus K)\times \mathbf{C}^{k}, g_{\mathrm{st}}, \nabla_{\mathrm{st}}) = (E_1|_{{M} \setminus K}, g^{E_1}, \nabla^{E_1}).
\end{equation}

Following \cite[Theorem 1.17]{GL83}, we choose a fixed point $x_0\in M$ and let $d:M\to \bf{R}^+$ be a regularization of the distance function ${\rm dist}(x,x_0)$ such that
\begin{equation*}
  \label{eq:fun-d}
  |\nabla d|(x)\leq 3/2,
\end{equation*}
for any $x\in M$.

Set
\begin{equation*}\label{728}
  B_m=\{x\in M: d(x)\leq m\},\ m\in \bf{N}
\end{equation*}
and choose a sufficiently large $m$ such that $K\subset B_m$.
Note that $B_m$ is compact due to the completeness of $g^{TM}$.

To construct the desired closed manifold $\widehat{M}_{H_{3m}}$, following \cite{GL83}, we take a compact hypersurface $H_{3m}\subseteq M\setminus B_{3m}$, which cuts $M$ into two parts such that the compact part, denoted by $M_{H_{3m}}$, contains $B_{3m}$. Then $M_{H_{3m}}$ is a compact smooth manifold with boundary $H_{3m}$.
Let $g^{TH_{3m}}$ be the induced metric on $H_{3m}$.
For a sufficiently small $\varepsilon'>0$, on the product manifold $H_{3m} \times [-\varepsilon',1+ \varepsilon']$, we construct a metric as follows.

Near the boundary $H_{3m}\times \{-\varepsilon'\}$ of $H_{3m} \times [-\varepsilon', 1+ \varepsilon']$, i.e.,\ on $H_{3m} \times [-\varepsilon', \varepsilon']$, by using the geodesic normal coordinate of $H_{3m} \subseteq M$, for a small $\varepsilon'$, there is an isomorphism between $H_{3m} \times [-\varepsilon', \varepsilon']$ and a neighborhood of $H_{3m}$, denoted by $U$, in $M$ because $H_{3m}$ is compact.
Moreover, we can require that under this isomorphism, $U\cap M_{H_{3m}}$ is mapped to $H_{3m}\times [-\varepsilon', 0]$.

Now, we define the metric on $H_{3m} \times [-\varepsilon', \varepsilon']$ to be the pull-back metric obtained from that of $U$.
In the same way, we can construct a metric on $H_{3m} \times [1-\varepsilon', 1+ \varepsilon']$.
Meanwhile, the metric on $H_{3m}\times [1/3,2/3]$ is defined to be the product metric of $g^{TH_{3m}}$ and the standard metric on $[1/3,2/3]$. 
Finally, the metric on $H_{3m} \times [-\varepsilon',1+\varepsilon']$ is a smooth extension of the metrics on the above three pieces. 

Using the isometry between $H_{3m} \times [-\varepsilon', \varepsilon']$ and $U$, $M_{H_{3m}} \cup U$ and $H_{3m} \times [-\varepsilon',1+\varepsilon']$ can be glued into a smooth Riemannian manifold with boundary,
\begin{equation*}
  (M_{H_{3m}} \cup U) \bigcup_{U} (H_{3m} \times [-\varepsilon',1+\varepsilon']).
\end{equation*}
More precisely, the resulting manifold is a quotient space of $(M_{H_{3m}} \cup U) \bigcup (H_{3m} \times [-\varepsilon',1+\varepsilon'])$, two points of $M_{H_{3m}} \cup U$ and $H_{3m} \times [-\varepsilon',1+\varepsilon']$ respectively correspond to the same point in the resulting manifold if and only if the two points are related by the isomorphism between $U$ and $H_{3m} \times [-\varepsilon',\varepsilon']$.

Let $M'_{H_{3m}}$ be another copy of $M_{H_{3m}}$ with the same metric and the opposite orientation.
In a similar way, we can glue $M'_{H_{3m}}$ and $(M_{H_{3m}} \cup U) \cup_{U} (H_{3m} \times [-\varepsilon',1+\varepsilon'])$ to obtain a closed manifold $\widehat{M}_{H_{3m}}$.
Note that $M_{H_{3m}} \cup U$, $M'_{H_{3m}}$ and $H_{3m}\times [1/3,2/3]$ all have natural isometric embeddings into $\widehat{M}_{H_{3m}}$.
As a result, we will view these manifolds as submanifolds of $\widehat{M}_{H_{3m}}$ in the following.
Figure~\ref{fig:glue} is an illustration of this gluing construction.

\begin{figure}[htbp]
  \centering
  \includegraphics[scale=1.2]{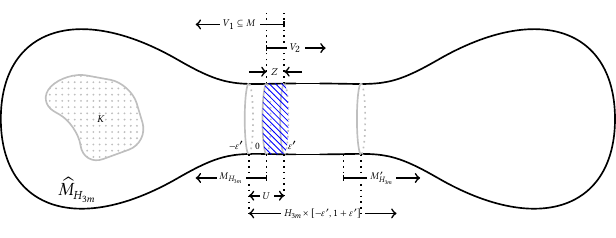}
  \caption{Gluing construction.}
  \label{fig:glue}
\end{figure}

Take $V_1 \coloneqq M_{H_{3m}} \cup U^{\circ}$ and $V_2 = \widehat{M}_{H_{3m}} \setminus M_{H_{3m}}$ to be two open subsets of $\widehat{M}_{H_{3m}}$.
Then
\begin{equation*}
  V_1 \cup V_2 = \widehat{M}_{H_{3m}},\quad V_1 \cap V_2 = Z \coloneqq H_{3m} \times (0,\varepsilon').
\end{equation*}
As we have said, we will treat $V_1 = M_{H_{3m}} \cup U^{\circ}$ as a submanifold of both $M$ and $\widehat{M}_{H_{3m}}$.
This means that, although $E$ is a $\mathbf{Z}_2$-graded vector bundle defined over $M$, its restriction on ${M_{H_{3m}}\cup U^{\circ}}$, that is $E|_{M_{H_{3m}}\cup U^{\circ}} = E|_{V_1}$, is a $\mathbf{Z}_2$-graded vector bundle defined over a submanifold of $\widehat{M}_{H_{3m}}$.
We are going to extend $E|_{V_1}$ to a $\mathbf{Z}_2$-graded bundle $\widehat{E} = \widehat{E}_0 \oplus \widehat{E}_1$ over $\widehat{M}_{H_{3m}}$ which satisfies
\begin{equation}
  \label{eq:e-ext}
  (\widehat{E}_0|_{V_2} , g^{\widehat{E}_0}, \nabla^{\widehat{E}_0}) \simeq (V_2 \times \mathbf{C}^k,g_{\mathrm{st}}, \nabla_{\mathrm{st}}) 
  \simeq (\widehat{E}_1|_{V_2}, g^{\widehat{E}_1}, \nabla^{\widehat{E}_1}).
\end{equation}

We will construct the $\widehat{E}_0$ and $\widehat{E}_1$ separately.
The construction of $\widehat{E}_1$ is straightforward.
Since
\begin{equation*}
  E_1|_{M_{H_{3m}}} = M_{H_{3m}} \times \mathbf{C}^k,
\end{equation*}
we can take $(\widehat{E}_1, g^{\widehat{E}_1}, \nabla^{\widehat{E}_1})$ to be $(\widehat{M}_{H_{3m}} \times \mathbf{C}^k,g_{\mathrm{st}}, \nabla_{\mathrm{st}})$, which satisfies (\ref{eq:e-ext}).

To construct $(\widehat{E}_0, g^{\widehat{E}_0}, \nabla^{\widehat{E}_0})$, we glue two vector bundles as in \cite{A}.
Choose the trivial bundle
\begin{equation*}
  (E'_0, g^{E'_0}, \nabla^{E'_0}) \coloneqq (V_2 \times \mathbf{C}^k, g_{\mathrm{st}}, \nabla_{\mathrm{st}})
\end{equation*}
over $V_2$.
Recall that $Z \subset U$ can be viewed as a submanifold of $M$ and $Z \cap K = \emptyset$.
Hence, by the definition of $\psi$, (\ref{eq:def-psi}), we have an isomorphism between $(E_0|_{Z}, g^{E_0}, \nabla^{E_0})$ (i.e., $((E_0|_{V_1})|_{Z}, g^{E_0}, \nabla^{E_0})$) and $(Z \times \mathbf{C}^k, g_{\mathrm{st}}, \nabla_{\mathrm{st}})$ (i.e., $(E'_0|_Z, g^{E'_0}, \nabla^{E'_0})$).
In other words, the restriction of $\psi$ on $Z$, denoted by $\psi|_Z$, induces an isomorphism between $((E_0|_{V_1})|_{Z}, g^{E_0}, \nabla^{E_0})$ and $(E'_0|_Z, g^{E'_0}, \nabla^{E'_0})$.
We define $\widehat{E}_0$ to be
\begin{equation*}
  \widehat{E}_0 \coloneqq E_0|_{V_1} \bigcup_{\psi|_Z} E'_0|_{V_2}.
\end{equation*}
By definition, $\psi$, thus $\psi|_Z$, preserves the metric and the connection.
As a result, $\widehat{E}_0$ inherits a metric and a connection from those of $E_0|_{V_1}$ and $E'_0|_{V_2}$.
Moreover, the property of gluing construction implies,
\begin{gather*}
  (\widehat{E}_0|_{M_{H_{3m}}}, g^{\widehat{E}_0}, \nabla^{\widehat{E}_0}) \simeq (E_0|_{M_{H_{3m}}}, g^{E_0}, \nabla^{E_0}),\\
  (\widehat{E}_0|_{V_2}, g^{\widehat{E}_0}, \nabla^{\widehat{E}_0}) \simeq (E'_0, g^{E'_0}, \nabla^{E'_0}) = (V_2 \times \mathbf{C}^k,g_{\mathrm{st}}, \nabla_{\mathrm{st}}).
\end{gather*}
Therefore, $\widehat{E}_0$ also satisfies (\ref{eq:e-ext}).

Let $i: E_0|_{V_1} \hookrightarrow \widehat{E}_0$ and $j: E'_0|_{V_2} \hookrightarrow \widehat{E}_0$ be the canonical embeddings in the gluing construction.
The definition of gluing construction implies the composition of maps
\begin{equation*}
  j|_Z^{-1} \circ i|_Z:(E_0|_{Z}, g^{E_0}, \nabla^{E_0}) \xhookrightarrow{i|_Z} (\widehat{E}_{0}|_Z, g^{\widehat{E}_0}, \nabla^{\widehat{E}_0}) \xhookrightarrow{j|_Z^{-1}} (E'_0|_Z,g_{\mathrm{st}}, \nabla_{\mathrm{st}}) = (Z \times \mathbf{C}^k, g_{\mathrm{st}}, \nabla_{\mathrm{st}})
\end{equation*}
is just $\psi|_Z$.
Since $\widehat{E}_1|_{V_2} = V_2 \times \mathbf{C}^k$, we can define
\begin{equation*}
  \nu\coloneqq j^{-1}: (\widehat{E}_0|_{V_2}, g^{\widehat{E}_0}, \nabla^{\widehat{E}_0}) \xrightarrow{\sim} (E'_0, g^{E'_0}, \nabla^{E'_0}) = (V_2 \times \mathbf{C}^k, g_{\mathrm{st}}, \nabla_{\mathrm{st}}) = (\widehat{E}_1|_{V_2}, g^{\widehat{E}_1}, \nabla^{\widehat{E}_1}).
\end{equation*}
At the same time, $\psi$ induces the map
\begin{multline*}
  \psi|_{V_1} \circ i^{-1}: (\widehat{E}_0|_{V_1}, g^{\widehat{E}_0}, \nabla^{\widehat{E}_0}) \xrightarrow{\sim} ({E}_0|_{V_1}, g^{E_0}, \nabla^{E_0}) \rightarrow ({E}_1|_{V_1}, g^{E_1}, \nabla^{E_1}) \\
  = (V_1 \times \mathbf{C}^k, g_{\mathrm{st}}, \nabla_{\mathrm{st}}) = (\widehat{E}_1|_{V_1}, g^{\widehat{E}_1}, \nabla^{\widehat{E}_1}).
\end{multline*}
Therefore, we have
\begin{equation}
  \label{eq:two-sec}
  (\psi|_{V_1} \circ i^{-1})|_Z = \psi|_Z \circ i|_Z^{-1} = j|_Z^{-1} \circ i|_Z \circ i|_Z^{-1} = j|_Z^{-1} = \nu|_Z.
\end{equation}
Note that the bundle endomorphism $\psi|_{V_1} \circ i^{-1}$ is just a smooth section of ${\rm Hom}(\widehat{E}_0,\widehat{E}_1)$ over $V_1$.
Similarly, the bundle endomorphism $\nu$ is a smooth section of ${\rm Hom}(\widehat{E}_0,\widehat{E}_1)$ over $V_2$.
Consequently, by (\ref{eq:two-sec}), the following section $\omega$ of ${\rm Hom}(\widehat{E}_0,\widehat{E}_1)$ is well defined and smooth,
\begin{equation*}
  \omega_x =
  \begin{cases}
    (\psi|_{V_1} \circ i^{-1})_x, & x \in V_1, \\
    \nu_x, & x \in V_2.
  \end{cases}
\end{equation*}
Moreover, by the property of $\psi$ and $\nu$, we know that $\omega$ preserves the metrics and connections on $\widehat{M}_{H_{3m}} \setminus K$.

Take $\omega^*$ to be the adjoint of $\omega$ with respect to $g^{\widehat{E}_0}$ and $g^{\widehat{E}_1}$.
Set 
\begin{equation*}\label{0.5}
  W=\omega+\omega^*: \Gamma(\widehat{E})\rightarrow \Gamma(\widehat{E}),
\end{equation*}
which is an odd and self-adjoint bundle endomorphism of $\widehat{E}$.
There exists a constant $\delta>0$ such that 
\begin{equation}\label{2.11a}
  W^2\geq \delta \text{ on } \widehat{M}_{H_{3m}} \setminus K.
\end{equation}

After the construction of $\widehat{M}_{H_{3m}}$ and $\widehat{E}$, we now explain how to construct a fiber bundle $\widehat{\mM}_{\mH_{3m},R}$ over $\widehat{M}_{H_{3m}}$ associated with the foliation $F$.
Let $F^\perp$ be the orthogonal complement to $F$, i.e., we have the orthogonal splitting 
 \begin{equation}\label{0.3}
TM=F\oplus F^\perp,\ \ \ g^{TM}=g^F\oplus g^{F^\perp}.
\end{equation}

Following  \cite[\S 5]{Co86} (also cf. \cite[\S 2.1]{Z17}), let $ \pi:\mM\rightarrow M$ be the Connes
fibration over $M$ such that for any $x\in M$, $\mM_{x}=\pi^{-1}(x)$
is the space of Euclidean metrics on the linear space $T_xM/F_{x}$.
Let  $T^V\mM$ denote the vertical tangent bundle of the fibration
$\pi:\mM\rightarrow M$. Then it carries a natural metric
$g^{T^V\mM}$ such that   any two points $p,\,q\in \mM_{x}$ with $x\in M$ can be joined by a unique geodesic along $\mM_{x}$. Let $d^{\mM_{x}}(p,q)$ denote the length of this geodesic.  
  
By  using the Bott connection
  on $TM/F$ (cf. \cite[(1.2)]{Z17}), which is leafwise flat, one  lifts $F$ to an integrable subbundle
$\mF$ of $T\mM$. 
  Then $g^{F}$   lifts to a Euclidean metric $g^{\mF}= \pi^* g^{F} $ on $\mF$.

  Let $\mF_{1}^\perp\subseteq T\mM$ be a subbundle, which is  transversal to $\mF\oplus T^V\mM$,   such that we have a
  splitting
  \begin{equation*}
    T\mM=(\mF \oplus T^V \mM)\oplus\mF_{1}^\perp.
  \end{equation*}
Then
$\mF_{1}^\perp$ can be identified with $T\mM/(\mF \oplus T^V \mM)$
and carries a canonically induced metric $g^{\mF_{1}^\perp}$.
We let $\mF_{2}^\perp$ denote $T^V\mM$.

  The metric $g^{F^\perp}$   in (\ref{0.3}) determines a canonical embedded section $s: M\hookrightarrow \mM$.
For any $p\in\mM$, set $\rho(p)=d^{\mM_{\pi(p)}}(p,s(\pi(p) ))$.

For any $  \beta,\ \gamma>0$,  following  \cite[(2.15)]{Z17}, let $g_{\beta,\gamma}^{T\mM}$ be the   metric    on $T\mM$  defined by
the
  orthogonal splitting,
\begin{equation}\label{0.6}\begin{split}
       T\mM =   \mF\oplus \mF^\perp_{1}\oplus \mF^\perp_{2},  \ 
\  \  \
g^{T\mM}_{\beta,\gamma}= \beta^2   g^{\mF}\oplus\frac{
g^{\mF^\perp_{1}}}{ \gamma^2 }\oplus g^{\mF^\perp_{2}}.\end{split}
\end{equation}

For any $R>0$, let $ \mM_{R}$ be the smooth manifold with boundary
defined by
\begin{equation*}\label{0.7} 
\mM_{R}=\left\{ p\in \mM\ :\  \rho(p)\leq R \right\}.
\end{equation*}

Set $\mH_{3m}= \pi^{-1}(H_{3m})$ and 
\begin{equation*}\label{0.8} 
\mM_{\mH_{3m},R} = \left(\pi^{-1}
\left(M_{H_{3m}}\right)\right)\cap \mM_{R},\; { \mH_{3m,R} = \mH_{3m}\cap \mM_{R}}.
\end{equation*}

Consider another copy $\mM_{\mH_{3m},R} '$ of $\mM_{\mH_{3m},R} $ carrying the metric $g^{T\mM'_{3m,R}}$ defined by (\ref{0.6}) with $\beta=\gamma=1$.
Meanwhile, let $g^{T\mH_{3m,R}}$ be the induced metric on $\mH_{3m,R}$ by (\ref{0.6}) with $\beta=\gamma=1$ and ${\rm d}t^2$ be the standard metric on $[1/3,2/3]$.
As we have done for $\widehat{M}_{H_{3m}}$, we can glue $\mM_{\mH_{3m},R} $, $\mM_{\mH_{3m},R} '$ and $\mH_{3m,R}\times [-\varepsilon', 1+ \varepsilon']$ together to get a manifold $\widehat{\mM}_{\mH_{3m},R}$, cf.~\cite[Section~2.2]{SWZ}.
But, unlike $\widehat{M}_{H_{3m}}$, $\widehat{\mM}_{\mH_{3m},R}$ is a smooth manifold with boundary.
Moreover, we can define a smooth metric $g^{T\widehat{\mM}_{\mH_{3m},R}}$ on $\widehat{\mM}_{\mH_{3m},R}$ such that
\begin{equation*}
  \label{eq:metric}
  \begin{aligned}[]
    g^{T\widehat{\mM}_{\mH_{3m},R}}|_{\mM_{\mH_{3m},R}} &= g^{T\mM_{3m,R}}_{\beta,\gamma},\\
    g^{T\widehat{\mM}_{\mH_{3m},R}}|_{\mM'_{\mH_{3m},R}} &= g^{T\mM'_{3m,R}},\\
    g^{T\widehat{\mM}_{\mH_{3m},R}}|_{\mH_{3m,R}\times [1/3,2/3]} &= g^{T\mH_{3m,R}}\oplus {\dd t^2}.
  \end{aligned}
\end{equation*}

The map $\pi:\mM_{\mH_{3m},R}\to M_{H_{3m}}$ can be extended to $\widehat{\mM}_{\mH_{3m},R}\to \widehat{M}_{H_{3m}}$ and we still denote the extended map by $\pi$.
As before, we pull the bundles on $\widehat{M}_{H_{3m}}$ back to $\widehat{\mM}_{\mH_{3m},R}$, that is, we take
\begin{equation*}
  (\mE_{3m,R},\nabla^{\mE_{3m,R}},g^{\mE_{3m,R}}) = \pi^*(\widehat{E},\nabla^{\widehat{E}},g^{\widehat{E}}).
\end{equation*}
As usual, $R^{\mE_{3m,R}}=(\nabla^{\mE_{3m,R}})^2$ is the curvature of $\nabla^{\mE_{3m,R}}$.

\subsubsection*{Step 2.}

Recall that we have assumed that $TM$ is oriented and spin, which implies that $\mF\oplus\mF^\perp_{1} =\pi^*(TM)$ is spin. Without loss of generality, as in~\cite[p. 1062-1063]{Z17}, we assume {further that} $F$ is oriented {and ${\rm rk}(F^\perp)$ is divisible by $4$}. Then $F^\perp$  is also oriented and $\dim \mM$ is even.

It is clear that $\mF\oplus \mF^\perp_{1},\, \mF^\perp_{2}$  over $\mM_{\mH_{3m},R}$ can be extended to $(\mH_{3m,R}\times [-\varepsilon', 1+ \varepsilon'])\cup \mM_{\mH_{3m},R}'$ such that we have the orthogonal splitting\footnote{$\mF$ restricted to $(\mH_{3m,R}\times [-\varepsilon', 1+ \varepsilon'])\cup \mM_{\mH_{3m},R}'$ needs no longer to be integrable.}
\begin{align}\label{0.10}
T\widehat \mM_{\mH_{3m},R}= \left(\mF\oplus\mF^\perp_{1} \right)\oplus \mF^\perp_{2}\ \  {\rm on}\ \ \widehat \mM_{\mH_{3m},R}.
\end{align}

Let $S_{\beta,\gamma}(\mF\oplus \mF^\perp_{1})$ denote the spinor bundle over $\widehat \mM_{\mH_{3m},R}$ with respect to the metric  $g^{T\widehat\mM_{\mH_{3m},R}}|_{\mF \oplus\mF^\perp_{1}} $ (thus with respect to
$\beta^2g^{\mF}\oplus \frac{g^{\mF^\perp_{1}}}{\gamma^2}$ on $\mM_{\mH_{3m},R}$). Let $\Lambda^* (\mF_{2}^\perp )$ denote the exterior algebra bundle of $\mF_{2}^{\perp,*}$, with the   ${\bf Z}_2$-grading given by the natural even/odd parity.

Let
\begin{equation}
  D _{\mF\oplus\mF_{1}^\perp,\beta,\gamma}:\Gamma (S_{\beta,\gamma} (\mF\oplus\mF_{1}^\perp)\widehat\otimes \Lambda^* (\mF_{2}^\perp ) ) \rightarrow \Gamma (S_{\beta,\gamma} (\mF\oplus\mF_{1}^\perp)\widehat \otimes \Lambda^* (\mF_{2}^\perp ) )
\end{equation}
be the sub-Dirac operator on $\widehat \mM_{\mH_{3m},R}$ constructed as in \cite[(2.16)]{Z17}.  
It is clear that one can canonically define the twisted sub-Dirac operator (twisted by $\mE_{3m,R}$) on $\widehat \mM_{\mH_{3m},R}$,
\begin{multline}\label{0.11}
D ^{\mE_{3m,R}}_{\mF\oplus\mF_{1}^\perp,\beta,\gamma}:\Gamma  (S_{\beta,\gamma}  (\mF\oplus\mF_{1}^\perp)\widehat\otimes
\Lambda^* (\mF_{2}^\perp )\widehat \otimes \mE_{3m,R})
\\
\rightarrow
\Gamma  (S_{\beta,\gamma}  (\mF\oplus\mF_{1}^\perp)\widehat\otimes
\Lambda^*(\mF_{2}^\perp )\widehat \otimes \mE_{3m,R} ).
\end{multline}

Let $\widetilde{f}:[0,1]\rightarrow [0,1]$ be a smooth function such that  $\widetilde{f}(t)= 0$ for $0\leq t\leq \frac{1}{4}$, while $\widetilde{f}(t) =1$ for $   \frac{1}{2}\leq t\leq 1$.  
For any $p\in \mM_{\mH_{3m},R}$, we connect $p$ and $s(\pi(p))$ by the unique geodesic in $\mM_{ \pi(p)}$. Let $\sigma(p)\in \mF_{2}^\perp|_p$ denote the unit vector tangent to this geodesic. Then  
\begin{align*}\label{0.13}
 \widetilde \sigma =\widetilde{ f}\left(\frac{\rho}{R}\right)\sigma
\end{align*}
is a smooth section of $\mF_{2}^\perp|_{\mM_{\mH_{3m},R}}$. It extends to a smooth section of $\mF_{2}^\perp|_{\widehat\mM_{\mH_{3m},R}}$, which we still denote by $\widetilde\sigma$.  It is easy to see that we may and we will assume that $\widetilde\sigma$ is transversal to (and thus nowhere zero on) $\partial \widehat\mM_{\mH_{3m},R}$.
Note that the Clifford action $\widehat c(\widetilde\sigma)$ (cf. \cite[(1.47)]{Z17}) now acts on $S_{\beta,\gamma}   (\mF\oplus\mF_{1}^\perp )\widehat\otimes
\Lambda^* (\mF_{2}^\perp  )\widehat \otimes \mE_{3m,R} $ over $\widehat\mM_{\mH_{3m},R}$.

With $\widehat c(\widetilde\sigma)$ and $W$, for $\varepsilon>0$, we introduce the following deformation of  $D ^{\mE_{3m,R}}_{\mF\oplus\mF_{1}^\perp,\beta,\gamma}$ on 
$\widehat\mM_{\mH_{3m},R}$, which combines the deformations in \cite[(2.21)]{Z17} and \cite[(1.11)]{Z19},

\begin{equation}
  \label{eq:defd}
 D ^{\mE_{3m,R}}_{\mF\oplus\mF_{1}^\perp,\beta,\gamma}
+\frac{\widehat c(\widetilde\sigma)}{\beta}
+\frac{\varepsilon \pi^*W}{\beta}.
\end{equation}

For the deformed operator (\ref{eq:defd}), the following estimate holds, which is an analog of \cite[Lemma 2.1]{SWZ}.
Let $h:[0,1]\rightarrow [0,1]$ be a smooth function such that $h(t)=1$ for $0\leq t\leq \frac{3}{4}$, while $h(t)=0$ for $\frac{7}{8}\leq t\leq 1$. 

\begin{lemma}\label{t0.4}
  There exist $c_0>0$, $\varepsilon>0$, $m>0$ and $R>0$ such that when $\beta,\,\gamma>0$ are small enough (which may depend on $m$ and $R$), 
  \begin{enumerate}[label=\emph{({\roman*})}]
  \item for any $s\in \Gamma  (S_{\beta,\gamma}   (\mF\oplus\mF_{1}^\perp )\widehat\otimes
  \Lambda^* (\mF_{2}^\perp )\widehat \otimes \mE_{3m,R}  )$ supported in the interior of $\widehat\mM_{\mH_{3m},R}$,\footnote{The norms below   depend on $\beta$ and $\gamma$. In case of no confusion, we omit the subscripts for simplicity.}
  \begin{equation*}\label{0.16}
    \Big\|\Big(D ^{\mE_{3m,R}}_{\mF\oplus\mF_{1}^\perp,\beta,\gamma}
    +\frac{\widehat c(\widetilde\sigma)}{\beta}
    +\frac{\varepsilon \pi^*W}{\beta}\Big)s\Big\|\geq \frac{c_0 }{\beta}\|s\|;
  \end{equation*}
  \item for any $s\in \Gamma  (S_{\beta,\gamma}   (\mF\oplus\mF_{1}^\perp )\widehat\otimes
  \Lambda^* (\mF_{2}^\perp )\widehat \otimes \mE_{3m,R}  )$
  supported in the interior of $\mM_{\mH_{3m},R}\setminus  \mM_{\mH_{3m},{R}/{2}}$,
  \begin{equation*}\label{0.17}
    \Big\|\Big(h\left(\frac{\rho}{R}\right) D ^{\mE_{3m,R}}_{\mF\oplus\mF_{1}^\perp,\beta,\gamma}
    h\left(\frac{\rho}{R}\right) 
    +\frac{\widehat c\left(\widetilde\sigma\right)}{\beta}
    +\frac{\varepsilon \pi^*W}{\beta}\Big)s\Big\|\geq \frac{c_0 }{\beta}\|s\|.
  \end{equation*}
    \end{enumerate}
\end{lemma}

\begin{proof}
  We follow the same strategy of~\cite[Lemma~2.1]{SWZ} to prove this lemma.
  Especially, the proof of (ii) is the same with the proof of~\cite[Lemma~2.1 (ii)]{SWZ}.
  Here, we only show how to modify the arguments in~\cite[Lemma~2.1]{SWZ} to prove (i).

  As in~\cite[(2.14)--(2.15)]{SWZ}, on $\widehat{M}_{H_{3m}}$, by using the regularized distance function $d(x)$, we can find cut-off functions $\psi_{m,1},\, \psi_{m,2}: \widehat{M}_{H_{3m}}\rightarrow [0,1]$ satisfying for $i=1,2$,
  \begin{gather}
    \psi^2_{m,1} + \psi^2_{m,2} = 1,\label{eq:sum}\\
    |\nabla \psi_{m,i}|(x)\leq {C/m} \text{ for any }x\in \widehat{M}_{H_{3m}}\label{eq:est-psi}.
  \end{gather}
  Then we pull $\psi_{m,1},\psi_{m,2}$ back to cut-off functions $\varphi_{m,1},\varphi_{m,2}$ defined on $\widehat \mM_{\mH_{3m},R}$.
  Due to~\cite[(2.24)]{SWZ}, we know that
  \begin{equation}
    \label{0.22}
    \begin{gathered}[]
      \varphi_{m,1}=1 \text{ if } x\in \pi^{-1}(B_m) \text{ and } \varphi_{m,1}=0 \text{ if } x\in \widehat{\mM}_{\mH_{3m},R}\setminus \pi^{-1}(B_{2m}),\\
      \varphi_{m,2}=0 \text{ if } x\in \pi^{-1}(B_m) \text{ and } \varphi_{m,2}=1 \text{ if } x\in \widehat{\mM}_{\mH_{3m},R}\setminus \pi^{-1}(B_{2m}). \\
    \end{gathered}
  \end{equation}

  Using $\varphi_{m,i}$, for any $s\in \Gamma (S_{\beta,\gamma} (\mF\oplus\mF_{1}^\perp )\widehat\otimes \Lambda^* (\mF_{2}^\perp )\widehat \otimes \mE_{3m,R} )$ supported in the interior of $\widehat\mM_{\mH_{3m},R}$, by (\ref{eq:sum}), we have the following estimate, cf.~\cite[(2.26)]{SWZ},
  \begin{multline}\label{0.22a}
    \sqrt{2}\Big\|\Big(D ^{\mE_{3m,R}}_{\mF\oplus\mF_{1}^\perp,\beta,\gamma}+\frac{\widehat c(\widetilde\sigma)}{\beta} + \frac{\varepsilon \pi^*W}{\beta}\Big)s\Big\| \ge 
    \Big\|\Big(D ^{\mE_{3m,R}}_{\mF\oplus\mF_{1}^\perp,\beta,\gamma}+\frac{\widehat c(\widetilde\sigma)}{\beta} + \frac{\varepsilon \pi^*W}{\beta}\Big)(\varphi_{m,1}s)\Big\| \\
    + \Big\| \Big(D ^{\mE_{3m,R}}_{\mF\oplus\mF_{1}^\perp,\beta,\gamma}+\frac{\widehat c(\widetilde\sigma)}{\beta} + \frac{\varepsilon \pi^*W}{\beta}\Big)(\varphi_{m,2}s)\Big\|
    -\|c_{\beta,\gamma}\left({\rm d}\varphi_{m,1}\right)s\|
    -\|c_{\beta,\gamma}\left({\rm d}\varphi_{m,2}\right)s\|,
  \end{multline}
  where for each $i\in \{1,2\}$, we identify ${\dd \varphi_{m,i}}$ with the gradient of $\varphi_{m,i}$ and $c_{\beta,\gamma}(\cdot)$ means the Clifford action with respect to the metric (\ref{0.6}).

  For the last two terms in the right-hand side of (\ref{0.22a}), we can use the estimate~\cite[(2.30)]{SWZ},
  \begin{equation}\label{2.19}
    |c_{\beta,\gamma}(\dd \varphi_{m,i})s|(x) =\Big(O\left({1\over {\beta m}}\right)+O_{m,R}(\gamma)\Big) |s|(x), \quad x\in\mM_{\mH_{3m},R},
  \end{equation}
  where the subscripts in $O_{m,R}(\cdot)$ mean that the big O constant may depend on $m$ and $R$.

  For the first two terms in the right-hand side of (\ref{0.22a}), by a direct computation, we have
  \begin{multline}\label{721}
    \Big(D ^{\mE_{3m,R}}_{\mF\oplus\mF_{1}^\perp,\beta,\gamma}+\frac{\widehat c(\widetilde\sigma)}{\beta}
    +\frac{\varepsilon \pi^*W}{\beta}\Big)^2=\Big(D ^{\mE_{3m,R}}_{\mF\oplus\mF_{1}^\perp,\beta,\gamma}+\frac{\widehat c(\widetilde\sigma)}{\beta}   \Big)^2 \\
    +\Big[D ^{\mE_{3m,R}}_{\mF\oplus\mF_{1}^\perp,\beta,\gamma},{{\varepsilon \pi^*W}\over{\beta}}\Big]
    +\frac{\varepsilon^2 (\pi^*W)^2}{\beta^2}.
  \end{multline} 
  Meanwhile, since $W$ is a constant endomorphism, i.e.\, parallel with respect to the connection, outside $K$, we know
  \begin{equation}\label{2.27a}
    \Big[D^{\mE_{3m,R}}_{\mF\oplus\mF_{1}^\perp,\beta,\gamma},{{\varepsilon \pi^*W}\over{\beta}}\Big]= 0\text{ on }\widehat{\mM}_{\mH_{3m},R}\setminus {\pi^{-1}(K)}.
  \end{equation}
  Therefore, for the second term on the right-hand side of (\ref{0.22a}), by (\ref{2.11a}), (\ref{0.22}), (\ref{721}) and (\ref{2.27a}), one has
  \begin{multline}\label{2.33}
    \Big\|\Big(D ^{\mE_{3m,R}}_{\mF\oplus\mF_{1}^\perp,\beta,\gamma}+\frac{\widehat c(\widetilde\sigma)}{\beta}
    +\frac{\varepsilon \pi^*W}{\beta}\Big)(\varphi_{m,2}s)\Big\|^2\\
    = \Big\|\Big(D ^{\mE_{3m,R}}_{\mF\oplus\mF_{1}^\perp,\beta,\gamma}+\frac{\widehat c(\widetilde\sigma)}{\beta}\Big)(\varphi_{m,2}s)\Big\|^2 +
    \Big\|\frac{\varepsilon \pi^*W}{\beta}(\varphi_{m,2}s)\Big\|^2
    \geq {{\delta\varepsilon^2}\over{\beta^2}}\left\|\varphi_{m,2}s\right\|^2.
  \end{multline}

  The main difference between the proof of (i) and~\cite[Lemma~2.1 (i)]{SWZ} lies in the estimate of the first term in the right-hand side of (\ref{0.22a}).
  Let $\mathrm{rk}(F)=\mathrm{rk}(\mathcal{F})=q$, $\mathrm{rk}(\mathcal{F}_1^{\perp})= q_1$ and $\mathrm{rk}({\mathcal{F}_2^{\perp}}) = q_2$.
  Since on $\mM_{\mH_{3m},R}$, $g^{\mF} = \pi^* g^F$, for a local orthonormal basis $\{f_1,\dots,f_q\}$ of $(\mF, g^{\mF})$, we can choose it to be lifted from a local orthonormal basis of $(F, g^F)$.
  Moreover, we choose $h_1, \dots, h_{q_1}$ (resp.\ $e_1, \dots, e_{q_2}$) to be a local orthonormal basis of $(\mF^{\perp}_1, g^{\mF^{\perp}_1})$ (resp.\ $(\mF^{\perp}_2, g^{\mF^{\perp}_2})$).
  Then,
  \begin{equation*}
    \label{eq:loc-fr}
    \{f_1,\dots,f_q,h_1,\dots,h_{q_1},e_1,\dots,e_{q_2}\}
  \end{equation*}
  is a local orthonormal frame for $T\mM_{\mH_{3m},R}$.
  Then, by~\cite[(2.40)]{SWZ}, we have
  \begin{multline}\label{2.35j}
    \Big\|\big(D ^{\mE_{3m,R}}_{\mF\oplus\mF_{1}^\perp,\beta,\gamma}+\frac{\widehat c(\widetilde\sigma)}{\beta}
    +\frac{\varepsilon \pi^*W}{\beta}\Big)(\varphi_{m,1}s)\Big\|^2 \ge \Big({{\pi^*k^F}\over{4\beta^2}}\varphi_{m,1}s, \varphi_{m,1}s\Big) \\
    +\Big({1\over{2\beta^2}}\sum_{i,j=1}^{q}R^{\mE_{3m,R}}(f_i,f_j)c_{\beta,\gamma}(\beta^{-1}f_i)c_{\beta,\gamma}(\beta^{-1}f_j)\varphi_{m,1}s, \varphi_{m,1}s\Big) \\
    +\Big(\Big[D ^{\mE_{3m,R}}_{\mF\oplus\mF_{1}^\perp,\beta,\gamma},\frac{\widehat c(\widetilde\sigma)}{\beta}\Big]\varphi_{m,1}s, \varphi_{m,1}s\Big)
    +\Big(\Big[D ^{\mE_{3m,R}}_{\mF\oplus\mF_{1}^\perp,\beta,\gamma},{{\varepsilon \pi^*W}\over{\beta}}\Big]\varphi_{m,1}s,\varphi_{m,1}s\Big) \\
    +\Big( \frac{\varepsilon^2 (\pi^*W)^2}{\beta^2}\varphi_{m,1}s,\varphi_{m,1}s\Big)+ \Big(O_{m,R}\left({1\over \beta}+{\gamma^2\over \beta^2}\right)\varphi_{m,1}s, \varphi_{m,1}s\Big).
  \end{multline}

  To estimate the right-hand side of (\ref{2.35j}), we proceed term by term.
  \begin{enumerate}[label=(\alph*)]
  \item For the first two terms, by (\ref{eq:def-rn}) and (\ref{4.3}),
    \begin{multline}
      \label{eq:k-pos}
      \Big({1\over{2\beta^2}}\sum_{i,j=1}^{q}R^{\mE_{3m,R}}(f_i,f_j)c_{\beta,\gamma}(\beta^{-1}f_i)c_{\beta,\gamma}(\beta^{-1}f_j)\varphi_{m,1}s, \varphi_{m,1}s\Big)\\
      + \Big({\pi^*k^F\over{4\beta^2}}\varphi_{m,1}s,\varphi_{m,1}s\Big) \geq {\frac{\kappa}{4\beta^2}}\|\varphi_{m,1}s\|^2.
    \end{multline}
    Note that (\ref{eq:k-pos}) is the counterpart of~\cite[(2.37)--(2.38)]{SWZ} in our situation.
  \item For the third term, by~\cite[Lemma 2.1]{Z17}, on $\mM_{\mH_{3m},R}\setminus {s(M_{H_{3m}})}$, we have
    \begin{equation}\label{2.27}
      \Big[D ^{\mE_{3m,R}}_{\mF\oplus\mF_{1}^\perp,\beta,\gamma},\frac{\widehat c(\widetilde\sigma)}{\beta}\Big]=O_m\left({1\over \beta^2 R}\right)+O_{m,R}\left(1\over \beta\right).
    \end{equation}
  \item For the fourth term, since $\nabla^{\mE_{3m,R}}$ (resp.\ $\pi^*W$) is a  pull-back connection (resp.\ bundle endomorphism) via $\pi$, by~\cite[(2.34)]{SWZ} and (\ref{2.27a}), we have
    \begin{equation}\label{2.26a}
      \Big[D ^{\mE_{3m,R}}_{\mF\oplus\mF_{1}^\perp,\beta,\gamma},{{\varepsilon \pi^*W}\over{\beta}}\Big]=
      \begin{cases}
        O\left({\varepsilon\over \beta^2}\right)+O_{R}\left({\varepsilon\gamma\over \beta}\right) & \text{ on }\pi^{-1}(K),\\
        0,& \text{ on }\widehat{\mM}_{\mH_{3m},R}\setminus {\pi^{-1}(K)}.
      \end{cases}
    \end{equation}
  \item For the fifth term, by (\ref{2.11a}), we have
    \begin{equation}
      \label{eq:ft}
      \Big(\frac{\varepsilon^2 (\pi^*W)^2}{\beta^2}\varphi_{m,1}s,\varphi_{m,1}s\Big) \ge \frac{\varepsilon^2 \delta}{\beta^2}\|\varphi_{m,1}s\|^2_{\pi^{-1}(B_{2m}\setminus K)},
    \end{equation}
    where the subscript on the norm means the integral on ${\pi^{-1}(B_{2m}\setminus K)}$.
  \end{enumerate}
  
  Now, as in~\cite[Lemma~2.1]{SWZ}, we split every term on the right-hand side of (\ref{2.35j}) into integrals on $\pi^{-1}(K)$ and ${\pi^{-1}(B_{2m}\setminus K)}$ separately.
  At the same time, we choose $\varepsilon$ small enough that
  \begin{equation*}
    {\kappa\over 8\beta^2}\|s\|^2_{\pi^{-1}(K)} +O\left(\varepsilon\over \beta^2\right)\|s\|^2_{\pi^{-1}(K)}\geq 0.
  \end{equation*}
  Then, by (\ref{eq:k-pos})--(\ref{eq:ft}), we have
  \begin{multline}\label{2.34}
    \Big\|\Big(D ^{\mE_{3m,R}}_{\mF\oplus\mF_{1}^\perp,\beta,\gamma}+\frac{\widehat c(\widetilde{\sigma})}{\beta}
    +\frac{\varepsilon \pi^*W}{\beta}\Big)(\varphi_{m,1}s)\Big\|^2
    \geq \min\Big\{{\kappa\over 8},\delta\varepsilon^2\Big\}\frac{\|\varphi_{m,1} s\|^2}{\beta^2}\\
    +O_{R}\left(\varepsilon\gamma\over \beta\right)\|\varphi_{m,1} s\|^2_{\pi^{-1}(K)}+O_{m,R}\left(\gamma^2\over \beta^2\right)\|\varphi_{m,1}s\|^2\\
    +O_{m,R}\left(1\over \beta\right)\|\varphi_{m,1}s\|^2+O_{m}\left({1\over \beta^2 R}\right)\|\varphi_{m,1}s\|^2.
  \end{multline}
  By (\ref{0.22a}), (\ref{2.33}) and (\ref{2.34}), by taking $m$ sufficiently
  large and then taking $R$ sufficiently large, one finds that there exist
  $c_0,\varepsilon, m, R>0$ such that when $\beta, \gamma>0$ are small enough, the estimate in (i) of Lemma~\ref{t0.4} holds.
\end{proof}

\subsubsection*{Step 3.}

Let $\partial \widehat \mM_{\mH_{3m},R}$ bound another oriented manifold $  \mN_{3m,R}$ such that
\begin{equation*}
  \widetilde  \mN_{3m,R}=\widehat \mM_{\mH_{3m},R}\cup \mN_{3m,R}
\end{equation*}
is an oriented closed manifold. 
Let $g_{}^{T\widetilde \mN_{3m,R}}$ be a smooth metric on $T\widetilde \mN_{3m,R}$ such that
\begin{equation*}
  g_{}^{T\widetilde \mN_{3m,R}}
  |_{\widehat\mM_{\mH_{3m},R}}=g_{}^{T\widehat\mM_{\mH_{3m},R} }.
\end{equation*}
The existence of $g_{}^{T\widetilde \mN_{3m,R}}$ is clear.

Let $Q$ be a Hermitian vector bundle over $\widehat \mM_{\mH_{3m},R}$ such that
\begin{equation*}
  (S_{\beta,\gamma}  (\mF\oplus\mF_{1}^\perp)\widehat\otimes \Lambda^*(\mF_{2}^\perp )\widehat \otimes \mE_{m,R})_{-}\oplus Q
\end{equation*}
is a trivial vector bundle over $\widehat \mM_{\mH_{3m},R}$. Then 
\begin{equation*}
(S_{\beta,\gamma}  (\mF\oplus\mF_{1}^\perp)\widehat\otimes \Lambda^*(\mF_{2}^\perp )\widehat \otimes \mE_{m,R})_{+}\oplus Q
\end{equation*}
is also a trivial vector bundle near $\partial \widehat \mM_{\mH_{3m},R}$ under the identification $\widehat{c}(\widetilde{\sigma})+\pi^* \omega+{\rm Id}_Q$.

Since the above two vector bundles are both trivial near $\partial \widehat \mM_{\mH_{3m},R}$, by extending them via the trivial bundle over $\widetilde  \mN_{3m,R} \setminus \widehat \mM_{\mH_{3m},R}$, we get a $\mathbf{Z}_2$-graded Hermitian vector bundle $\xi=\xi_{+}\oplus \xi_-$ over $\widetilde  \mN_{3m,R}$ and an odd self-adjoint endomorphism $\mathcal{W}'=\omega'+\omega'^*\in \Gamma({\rm End}(\xi))$ (with $\omega':\Gamma(\xi_+)\to \Gamma(\xi_-)$, $\omega'^*$ being the adjoint of $\omega'$) such that
\begin{equation*}
  \xi_{\pm}= (S_{\beta,\gamma}  (\mF\oplus\mF_{1}^\perp)\widehat\otimes \Lambda^*(\mF_{2}^\perp )\widehat \otimes \mE_{m,R})_{\pm }\oplus Q
\end{equation*}
over $\widehat \mM_{\mH_{3m},R}$, $\mathcal{W}'$ is invertible on $\mN_{3m,R}$ and 
\begin{align}\label{0.27}
  \mW'=\widehat{c}(\widetilde{\sigma})+\pi^*W+
  \begin{pmatrix}0 & {\rm Id}_Q\\
    {\rm Id}_Q & 0\\
  \end{pmatrix}
\end{align}
on $\widehat \mM_{\mH_{3m},R}$, which is invertible on $\widehat \mM_{\mH_{3m},R}\setminus  \mM_{\mH_{3m},{R/ 2}}$.

Recall that $h({\rho}/{R})$ vanishes near $\mM_{\mH_{3m},R}\cap \partial  \mM_{R}$. We extend it to a function on $\widetilde \mN_{ 3m,R}$ which equals zero on $\mN _{3m,R}$ and an open neighborhood of $\partial \widehat \mM_{\mH_{3m},R}$ in $\widetilde \mN_{ 3m,R}$, and we denote the resulting function on $\widetilde \mN_{3m,R}$ by $\widetilde h_R$.

Let $\pi_{\widetilde \mN_{3m, R}}: T\widetilde \mN_{3m,R}\to \widetilde \mN_{3m, R}$ be the projection of the tangent bundle of $\widetilde \mN_{3 m,R}$. Let $\gamma^{\widetilde \mN_{3 m,R}}\in {\rm Hom} (\pi^*_{\widetilde \mN_{3m, R}}\xi_+,\pi^*_{\widetilde \mN_{ 3m,R}}\xi_-)$ be the symbol defined by 
\begin{align}\label{0.28}
  \gamma^{\widetilde \mN_{ 3m,R}}(p,u)=\pi^*_{\widetilde \mN_{3m, R}}\big(\sqrt{-1} \widetilde h ^2_R c_{\beta,\gamma}(u)+\omega'(p)\big)\ \ {\rm for}\ \ p\in \widetilde \mN_{ 3m,R},\ \ u\in T_p \widetilde \mN_{ 3m,R}.
\end{align}
By (\ref{0.27}) and (\ref{0.28}), $\gamma^{\widetilde\mN_{3 m,R}}$ is singular only if $u=0$ and $p\in\mM_{\mH_{3m},{R}/{2}}$. Thus $\gamma^{\widetilde \mN_{ 3m,R}}$ is an elliptic symbol.

On the other hand, it is clear that $\widetilde h_R  D ^{\mE_{3m,R}}_{\mF\oplus\mF_{1}^\perp,\beta,\gamma}\widetilde h_R$ is well defined on $\widetilde \mN_{ 3m,R}$ if we define it to be zero on $\widetilde \mN_{ 3m,R}\setminus \widehat \mM_{\mH_{3m},R}$.

Let $A: L^2 (\xi)\to L^2 (\xi)$ be a second order positive elliptic differential operator on $\widetilde \mN_{ m,R}$ preserving the ${\bf Z}_2$-grading of $\xi=\xi_+\oplus \xi_-$, such that its symbol equals to $|\eta|^2$ at $\eta\in T\widetilde \mN_{ 3m,R}$.\footnote{To be more precise, here $A$ also depends on the defining metric. We omit the corresponding subscript/superscript only for convenience.}
As in \cite[(2.33)]{Z17}, let $P^{\mE_{3m,R}}_{ R,\beta,\gamma}:L^2 (\xi)\to L^2 (\xi)$ be the zeroth order pseudodifferential operator on $\widetilde \mN_{3m, R}$ defined by
\begin{align}\label{0.29}
P^{\mE_{3m,R}}_{R,\beta,\gamma}=A^{-\frac{1}{4}}\widetilde h_R D ^{\mE_{3m,R}}_{\mF\oplus\mF_{1}^\perp,\beta,\gamma}\widetilde h_R A^{-\frac{1}{4}}+\frac{\mW'}{\beta}.
\end{align}
Let $P^{\mE_{3m,R}}_{R,\beta,\gamma,+}:L^2(\xi_+)\to L^2(\xi_-)$ be the obvious restriction.
Then the principal symbol of $P^{\mE_{3m,R}}_{R,\beta,\gamma,+}$, which we denote by $\gamma(P^{\mE_{3m,R}}_{R,\beta,\gamma,+})$, is homotopic through elliptic symbols to $\gamma^{\widetilde \mN_{3m, R}}$. Thus $P^{\mE_{3m,R}}_{R,\beta,\gamma,+}$ is a Fredholm operator.
Moreover, the index of the symbol $\gamma^{\widetilde \mN_{3m, R}}$ can be calculated by the Atiyah-Singer index theorem directly (cf. \cite{ASI} and \cite[Proposition III.
11.24]{LaMi89}).
Therefore,
\begin{multline}
  \label{eq:ind1}
{\rm ind}\Big(P^{\mE_{3m,R}}_{R,\beta,\gamma,+}\Big)={\rm ind}\Big(\gamma\big(P^{\mE_{3m,R}}_{R,\beta,\gamma,+}\big)\Big)
={\rm ind}\Big(\gamma^{\widetilde \mN_{ 3m,R}}\Big)\\
=\Big\langle\widehat{\mathrm{A}}(T\widehat{M}_{H_{3m}})({\rm ch}(\widehat{E}_0)-{\rm ch}(\widehat{E}_1)),[\widehat{M}_{H_{3m}}]\Big\rangle=\Big\langle\widehat{\mathrm{A}}(T{M})({\rm ch}(E_0)-{\rm ch}(E_1)),[M]\Big\rangle\neq 0,
\end{multline}
where the inequality comes from (\ref{84}).

For any $0\leq t\leq 1$, set
\begin{align}\label{0.31}
  P^{\mE_{3m,R}}_{R,\beta,\gamma,+}(t)=P^{\mE_{3m,R}}_{R,\beta,\gamma,+}+\frac{(t-1)\omega'}{\beta}+A^{-\frac{1}{4}}\frac{(1-t)\omega'}{\beta}A^{-\frac{1}{4}}.
\end{align}
Then $P^{\mE_{3m,R}}_{R,\beta,\gamma,+}(t)$ is a smooth family of zeroth order pseudodifferential operators such that the corresponding symbol $\gamma(P^{\mE_{3m,R}}_{R,\beta,\gamma,+}(t))$ is elliptic for $0<t\leq 1$. Thus $P^{\mE_{3m,R}}_{R,\beta,\gamma,+}(t)$ is a continuous family of Fredholm operators for $0<t\leq 1$ with $P^{\mE_{3m,R}}_{R,\beta,\gamma,+}(1)=P^{\mE_{3m,R}}_{R,\beta,\gamma,+}$.

Then, by using Lemma~\ref{t0.4}, the exactly same arguments in~\cite[Proposition 2.2]{SWZ} show that for suitable $\varepsilon, m, R, \beta, \gamma>0$,
\begin{equation*}\label{0.32}
  \dim\Big({\rm ker}\big(P^{\mE_{3m,R}}_{R,\beta,\gamma,+}(0)\big)\Big)=\dim\Big({\rm ker}\big(P^{\mE_{3m,R}}_{R,\beta,\gamma,+}(0)^*\big)\Big)=0.
\end{equation*}
As a result, when $t=0$, $P^{\mE_{3m,R}}_{R,\beta,\gamma,+}(0)$ is also Fredholm and has a vanishing index.
By the property of Fredholm index, we have
\begin{equation*}
  {\rm ind}{\Big(P^{\mE_{3m,R}}_{R,\beta,\gamma,+}\Big)} = {\rm ind}\Big(P^{\mE_{3m,R}}_{R,\beta,\gamma,+}(1)\Big)  = {\rm ind}\Big(P^{\mE_{3m,R}}_{R,\beta,\gamma,+}(0)\Big) = 0,
\end{equation*}
which contradicts (\ref{eq:ind1}) and the proof of Thoerem~\ref{thm:4.1} is completed.

\vspace{\baselineskip}

\noindent{\bf Acknowledgments.}
G.\ Su was partially supported by NSFC grant 12425106, 11931007 and 12271266, the Nankai Zhide Foundation and the Fundamental Research Funds for the Central Universities grant 63243068.
X.\ Wang was partially supported by NSFC grant 12471049 and 12101361, the Project of Young Scholars of Shandong University.

\bibliographystyle{amsplain}
\providecommand{\bysame}{\leavevmode\hbox to3em{\hrulefill}\thinspace}
\providecommand{\MR}{\relax\ifhmode\unskip\space\fi MR }
\providecommand{\MRhref}[2]{\href{http://www.ams.org/mathscinet-getitem?mr=#1}{#2}
}
\providecommand{\href}[2]{#2}

\end{document}